\documentclass[12pt,twoside]{amsart}
\usepackage{amsmath, amsthm, amscd, amsfonts, amssymb, graphicx}
\usepackage{enumerate}
\usepackage[colorlinks=true,
linkcolor=blue,
urlcolor=black,
citecolor=red]{hyperref}
\usepackage{mathrsfs}
\newtheorem{theorem}{Theorem}[section]
\newtheorem{lemma}{Lemma}[section]
\newtheorem{remark}{Remark}[section]

\newtheorem{corollary}{Corollary}[section]

\newtheorem{proposition}{Proposition}[section]
\numberwithin{equation}{section}

\begin{document}
	
\title{Further Properties of PPT  and Hyponormal Matrices}
\author{Hamid Reza Moradi, Ibrahim Halil G\"um\"u\c s, and  Mohammad Sababheh}

\begin{abstract}
This paper discusses further properties of positive partial transpose matrices, with applications towards hyponormal, semi-hyponormal, and $(\alpha,\beta)$-normal matrices. The obtained results present extensions and improvements of many results in the literature.
\end{abstract}

\maketitle
\pagestyle{myheadings}
\markboth{\centerline {}}
{\centerline {}}
\bigskip
\bigskip

\section{Introduction and Preliminaries}

In the sequel, $\mathcal{M}_n$ denotes the algebra of all $n\times n$ complex matrices. When $A\in\mathcal{M}_n$ is such that $\left<Ax,x\right>\geq 0$ for all $x\in\mathbb{C}^n$, we say that  $A$ is positive semidefinite, and we write $A\geq O$. On the other hand, if $\left<Ax,x\right>>0$ for all nonzero $x\in\mathbb{C}^n$, $A$ is said to be positive definite, and we write $A>O$. Here, $O$ denotes the zero element of $\mathcal{M}_n$.

Given $A,B,X\in\mathcal{M}_n$,  the matrix $\left[\begin{matrix}A&X^*\\ X&B\end{matrix}\right]$ is in $\mathcal{M}_{2n}.$ It is well known that if $\left[ \begin{matrix}
   A & {{X}^{*}}  \\
   X & B  \\
\end{matrix} \right]\geq O$,  then
\begin{equation}\label{8}
\left\| \left[ \begin{matrix}
   A & {{X}^{*}}  \\
   X & B  \\
\end{matrix} \right] \right\|\le \left\| A \right\|+\left\| B \right\|,
\end{equation}
where $\|\cdot\|$ is the usual operator norm. Indeed, \eqref{8} follows from the following useful decomposition \cite[Lemma 3.4]{2}: For every matrix $\left[ \begin{matrix}
   A & {{X}^{*}}  \\
   X & B  \\
\end{matrix} \right]\ge O$, we have 
	\[\left[ \begin{matrix}
   A & {{X}^{*}}  \\
   X & B  \\
\end{matrix} \right]=U\left[ \begin{matrix}
   A & O  \\
   O & O  \\
\end{matrix} \right]{{U}^{*}}+V\left[ \begin{matrix}
   O & O  \\
   O & B  \\
\end{matrix} \right]{{V}^{*}}\]
for some unitaries $U,V$.

However, if $\left[ \begin{matrix}
   A & {{X}^{*}}  \\
   X & B  \\
\end{matrix} \right]\ge O$ and the off-diagonal block $X$ is Hermitian,  Hiroshima \cite{7} showed a stronger inequality than \eqref{8}, as follows
\begin{equation}\label{9}
\left\| \left[ \begin{matrix}
   A & X  \\
   X & B  \\
\end{matrix} \right] \right\|\le \left\| A+B \right\|.
\end{equation}

The block matrix  $ \left[ \begin{matrix}
   A & {{X}^{*}}  \\
   X & B  \\
\end{matrix} \right]$ has received the attention of numerous researchers in the literature due to its usability and applications. For example,  $ \left[ \begin{matrix}
   A & {{X}^{*}}  \\
   X & B  \\
\end{matrix} \right]$ is said to be positive partial transpose (PPT) if both $\left[ \begin{matrix}
   A & X  \\
   {{X}^{*}} & B  \\
\end{matrix} \right]$ and $\left[ \begin{matrix}
   A & {{X}^{*}}  \\
   X & B  \\
\end{matrix} \right]$ are positive semidefinite.

If $\left[ \begin{matrix}
   A & {{X}^{*}}  \\
   X & B  \\
\end{matrix} \right]$ is PPT, Lee \cite[Theorem 2.1]{6} showed that
\begin{equation}\label{7}
\left| X \right|\le \frac{A\sharp B+{{U}^{*}}(A\sharp B)U}{2},
\end{equation}
where   $X=U\left| X \right|$ is the polar decomposition of $X$, and $\sharp$ is the matrix geometric mean. Recall, here, that if $A,B>O$ and $0<t<1$, the weighted geometric mean of $A$ and $B$ is defined by 
\begin{align*}
A\sharp_t B=A^{\frac{1}{2}}\left(A^{-\frac{1}{2}}BA^{-\frac{1}{2}}\right)^tA^{\frac{1}{2}}.
\end{align*}
When $t=\frac{1}{2},$ we simply write $A\sharp B$ instead of $A\sharp_{\frac{1}{2}}B$. The geometric mean $\sharp$ was first
introduced by Pusz and Woronowicz \cite{pw}, which was further developed into a general theory of operator means by Kubo and Ando \cite{ka}.

Fu et al. \cite[Theorem 2.3]{3}  improved  \eqref{7} as follows
\begin{equation}\label{17}
\left| X \right|\le \left( A\sharp B \right)\sharp \left( {{U}^{*}}\left( A\sharp B \right)U \right).
\end{equation}
The fact that \eqref{17} improves \eqref{7} follows from the arithmetic-geometric mean inequality that states 
\begin{align}\label{eq_amgm}
A\sharp B\leq \frac{A+B}{2},
\end{align}
for any $A,B\in\mathcal{M}_n$ with $A,B>O.$

Among those useful characterizations of the geometric mean, we have \cite[(4.15)]{11})
\begin{equation}\label{4}
X\sharp Y=\max \left\{ Z:\text{ }Z={{Z}^{*}},\left[ \begin{matrix}
   X & Z  \\
   Z & Y  \\
\end{matrix} \right]\ge O \right\}; X,Y>O.
\end{equation} 
Recently, the following lemma has been shown in \cite[Theorem 2.1]{10}.
\begin{lemma}\label{2}
 If $\left[ \begin{matrix}
   A & X  \\
   {{X}^{*}} & B  \\
\end{matrix} \right]$ is PPT, then so is $\left[ \begin{matrix}
   A{{\sharp }_{t}}B & X  \\
   {{X}^{*}} & A{{\sharp}_{1-t}}B  \\
\end{matrix} \right]$, $0\le t\le 1$.

\end{lemma}

Further, it has been shown in \cite[Corollary 2.2]{6} that  
\begin{equation}\label{eq_lee}
{{\lambda }_{j}}\left( 2\left| X \right|-A{{\sharp }_{t}}B \right)\le {{\lambda }_{j}}\left( A{{\sharp }_{1-t}}B \right),
\end{equation}
when $\left[ \begin{matrix}
   A & {{X}^{*}}  \\
   X & B  \\
\end{matrix} \right]$ is PPT. Here $\lambda_j$ denotes the $j^{\text{th}}$ largest eigenvalue.

This paper discusses extensions of \eqref{17} and \eqref{eq_lee}, where we extend both inequalities to the weighted geometric mean. We will also extend  \cite[Theorem 2.2]{3}, where we show that when $\left[ \begin{matrix}
   A & X  \\
   {{X}^{*}} & B  \\
\end{matrix} \right]\geq O$, then $\left| X \right|\le \left( A{{\sharp }_{t}}{{U}^{*}}BU \right)\sharp \left( A{{\sharp}_{1-t}}{{U}^{*}}BU \right)$, for example. Further, we give an improvement of \eqref{8} that is different from \eqref{9}. Many other consequences for PPT matrices and positive semidefinite matrices of the form $\left[ \begin{matrix}
   A & X  \\
   {{X}^{*}} & B  \\
\end{matrix} \right]$ will be presented.

After that, we present some applications that include semi-hyponormal and $(\alpha,\beta)$-normal matrices. Here we recall that if $T\in\mathcal{M}_n$ is such that $|T^*|^2\leq |T|^2$, then $T$ is said to be hyponormal. If $T$ satisfies the weaker condition $|T^*|\leq |T|$, then $T$ is said to be semi-hyponormal. More generally, if $0\le \alpha \le 1\le \beta$ are such that $\alpha^2|T|^2\leq |T^*|^2\leq\beta^2|T|^2$, then $T$ is said to be $(\alpha,\beta)$-normal. For example, we will show that\[\left[ \begin{matrix}
   \left| T \right| & {{T}^{*}}  \\
   T & \left| T \right|  \\
\end{matrix} \right]\ge O\] if and only if $T$ is semi-hyponormal. Many other results and consequences will be shown for these classes. As a consequence, we will be able to present a possible reverse for the inequality $\|T^2\|\leq \|T\|^2$, when $T$ is $(\alpha,\beta)$-normal.

\section{Main Results for positive and PPT block matrices}
Our first result is an extension of  \cite[Theorem 2.2]{3}, which states that if  $\left[ \begin{matrix}
   A & X  \\
   {{X}^{*}} & B  \\
\end{matrix} \right]\geq O$, then 
\[\left| X \right|\le A\sharp {{U}^{*}}BU\text{ and }\left| {{X}^{*}} \right|\le UA{{U}^{*}}\sharp B.\]
Once this has been shown, we use it to extend some results about PPT matrices.
\begin{theorem}\label{5}
Let  $\left[ \begin{matrix}
   A & X  \\
   {{X}^{*}} & B  \\
\end{matrix} \right]\geq O$ with $A,B,X\in {{\mathcal M}_{n}}$ and let $X=U\left| X \right|$ be the polar decomposition of $X$. Then  for any $0\le t\le 1$,
\begin{equation}\label{16}
\left| X \right|\le \left( A{{\sharp }_{t}}{{U}^{*}}BU \right)\sharp \left( A{{\sharp}_{1-t}}{{U}^{*}}BU \right)
\end{equation}
and
\begin{equation}\label{15}
\left| {{X}^{*}} \right|\le \left( UA{{U}^{*}}{{\sharp}_{t}}B \right)\sharp \left( UA{{U}^{*}}{{\sharp}_{1-t}}B \right).
\end{equation}
In particular,
\[\left| X \right|\le A\sharp {{U}^{*}}BU\text{ and }\left| {{X}^{*}} \right|\le UA{{U}^{*}}\sharp B.\]
\end{theorem}
\begin{proof}
We prove  \eqref{16}. Since $X=U\left| X \right|$ is the polar decomposition of  $X$, we have
\[\left[ \begin{matrix}
   I & O  \\
   O & {{U}^{*}}  \\
\end{matrix} \right]\left[ \begin{matrix}
   A & {{X}^{*}}  \\
   X & B  \\
\end{matrix} \right]\left[ \begin{matrix}
   I & O  \\
   O & U  \\
\end{matrix} \right]=\left[ \begin{matrix}
   A & {{X}^{*}}U  \\
   {{U}^{*}}X & {{U}^{*}}BU  \\
\end{matrix} \right]\ge O,\]
which implies
\[\left[ \begin{matrix}
   A & \left| X \right|  \\
   \left| X \right| & {{U}^{*}}BU  \\
\end{matrix} \right]\ge O,\]
since ${{X}^{*}}U={{U}^{*}}X=\left| X \right|$. Lemma \ref{2} gives
\[\left[ \begin{matrix}
   A{{\sharp}_{t}}{{U}^{*}}BU & \left| X \right|  \\
   \left| X \right| & A{{\sharp}_{1-t}}{{U}^{*}}BU  \\
\end{matrix} \right]\ge O.\]
Now, the result follows from \eqref{4}. To prove  \eqref{15}, we have
\[\begin{aligned}
   \left| {{X}^{*}} \right|&=U\left| X \right|{{U}^{*}} \quad \text{(by \cite[p. 58]{9})}\\ 
 & \le U\left( \left( A{{\sharp }_{t}}{{U}^{*}}BU \right)\sharp \left( A{{\sharp }_{1-t}}{{U}^{*}}BU \right) \right){{U}^{*}} \\ 
 & =U\left( A{{\sharp}_{t}}{{U}^{*}}BU \right){{U}^{*}}\sharp U\left( A{{\sharp}_{1-t}}{{U}^{*}}BU \right){{U}^{*}} \\ 
 & =\left( UA{{U}^{*}}{{\sharp }_{t}}BU \right)\sharp \left( UA{{U}^{*}}{{\sharp }_{1-t}}B \right).  
\end{aligned}\]
This proves \eqref{15}. Letting $t=\frac{1}{2}$ in \eqref{16} and \eqref{15}, and noting that $T\sharp T=T,$ when $T>O$, imply $\left| X \right|\le A\sharp {{U}^{*}}BU\text{ and }\left| {{X}^{*}} \right|\le UA{{U}^{*}}\sharp B.$ This completes the proof.
\end{proof}

Using Theorem \ref{5}, we present the following extension of \eqref{17}.
\begin{corollary}\label{6}
Let $A,B,X\in {{\mathcal M}_{n}}$ be  such that $\left[ \begin{matrix}
   A & {{X}^{*}}  \\
   X & B  \\
\end{matrix} \right]$ is PPT, and let $X=U\left| X \right|$ be the polar decomposition of  $X$. Then for any $0\le t\le 1$,
\[\left| X \right|\le \left( A{{\sharp}_{t}}B \right)\sharp \left( {{U}^{*}}\left( A{{\sharp }_{1-t}}B \right)U \right),\]
and
\[\left| {{X}^{*}} \right|\le \left( U\left( A{{\sharp }_{t}}B \right){{U}^{*}} \right)\sharp \left( A{{\sharp}_{1-t}}B \right).\]
\end{corollary}
\begin{proof}
Lemma \ref{2} ensures that $\left[ \begin{matrix}
   A{{\sharp}_{t}}B & {{X}^{*}}  \\
   X & A{{\sharp}_{1-t}}B  \\
\end{matrix} \right]$ is also PPT, for any $0\le t\le 1$. By Theorem \ref{5}, we conclude the first desired result.

The second inequality can be shown using the method we used to show  \eqref{15}. This completes the proof.
\end{proof}

Interestingly, Corollary \ref{6} can be used to present a weighted version of \eqref{eq_lee}, as follows.
\begin{corollary}\label{26}
Let $A,B,X\in {{\mathcal M}_{n}}$ be such that $\left[ \begin{matrix}
   A & {{X}^{*}}  \\
   X & B  \\
\end{matrix} \right]$ is PPT. Then for any $0\le t\le 1$,
\[{{\lambda }_{j}}\left( 2\left| X \right|-A{{\sharp }_{t}}B \right)\le {{\lambda }_{j}}\left( A{{\sharp }_{1-t}}B \right)\]
and
\[{{\lambda }_{j}}\left( 2\left| {{X}^{*}} \right|-A{{\sharp}_{1-t}}B \right)\le {{\lambda }_{j}}\left( A{{\sharp}_{t}}B \right)\]
for all $j=1,2,\ldots ,n$.
\end{corollary}
\begin{proof}
Corollary \ref{6} means that
\[\begin{aligned}
   \left| X \right|&\le \left( A{{\sharp }_{t}}B \right)\sharp \left( {{U}^{*}}(A{{\sharp }_{1-t}}B)U \right) \\ 
 & \le \frac{A{{\sharp }_{t}}B+{{U}^{*}}(A{{\sharp }_{1-t}}B)U}{2}.
\end{aligned}\]
Thus,
\[2\left| X \right|-A{{\sharp }_{t}}B\le {{U}^{*}}(A{{\sharp }_{1-t}}B)U.\]
Therefore,
\[{{\lambda }_{j}}\left( 2\left| X \right|-A{{\sharp }_{t}}B \right)\le {{\lambda }_{j}}\left( A{{\sharp }_{1-t}}B \right)\]
as desired.
\end{proof}

Related to the discussion, we employ \eqref{9} to obtain another refinement of  \eqref{8}, as follows.
\begin{theorem}\label{10}
Let $A,B,X\in {{\mathcal M}_{n}}$ be such that $\left[ \begin{matrix}
   A & {{X}^{*}}  \\
   X & B  \\
\end{matrix} \right]\ge O$ and let $X=U\left| X \right|$ be the polar decomposition of $X$. Then
\[\left\| \left[ \begin{matrix}
   A & {{X}^{*}}  \\
   X & B  \\
\end{matrix} \right] \right\|\le \left\| A+{{U}^{*}}BU \right\|.\]
\end{theorem}
\begin{proof}
If $X=U\left| X \right|$ is the polar decomposition of $X$, then $$\left[ \begin{matrix}
   I & O  \\
   O & {{U}^{*}}  \\
\end{matrix} \right]\left[ \begin{matrix}
   I & O  \\
   O & U  \\
\end{matrix} \right]=\left[ \begin{matrix}
   I & O  \\
   O & U  \\
\end{matrix} \right]\left[ \begin{matrix}
   I & O  \\
   O & {{U}^{*}}  \\
\end{matrix} \right]=\left[ \begin{matrix}
   I & O  \\
   O & I  \\
\end{matrix} \right],$$
since $U$ is unitary. The fact that $\left[ \begin{matrix}
   A & {{X}^{*}}  \\
   X & B  \\
\end{matrix} \right]\ge O$ implies
\[\begin{aligned}
   \left\| \left[ \begin{matrix}
   A & {{X}^{*}}  \\
   X & B  \\
\end{matrix} \right] \right\|&=\left\| \left[ \begin{matrix}
   I & O  \\
   O & {{U}^{*}}  \\
\end{matrix} \right]\left[ \begin{matrix}
   A & {{X}^{*}}  \\
   X & B  \\
\end{matrix} \right]\left[ \begin{matrix}
   I & O  \\
   O & U  \\
\end{matrix} \right] \right\| \\ 
 & =\left\| \left[ \begin{matrix}
   A & \left| X \right|  \\
   \left| X \right| & {{U}^{*}}BU  \\
\end{matrix} \right] \right\| \\ 
 & \le \left\| A+{{U}^{*}}BU \right\|, 
\end{aligned}\]
where we have used \eqref{9} to obtain the last inequality. This completes the proof.
\end{proof}

The following is an interesting characterization related to PPT matrices.
\begin{theorem}\label{12}
Let $A,B,X\in\mathcal{M}_n$ be such that $\left[\begin{matrix} A&X\\X^*&B\end{matrix}\right]$ is PPT and let $0\leq t\leq 1$. Then there are some isometries $\widetilde{U},\widetilde{V}\in {{\mathcal M}_{2n,n}}$ (depending on $t$), such that
$$\left[ \begin{matrix}
   A{{\sharp}_{t}}B & X  \\
   {{X}^{*}} & A{{\sharp}_{1-t}}B  \\
\end{matrix} \right]=\widetilde{U}\left( A{{\sharp}_{t}}B \right){{\widetilde{U}}^{*}}+\widetilde{V}\left( A{{\sharp}_{1-t}}B \right){{\widetilde{V}}^{*}}.$$
 
\end{theorem}
\begin{proof}
By Lemma \ref{2}, $\left[ \begin{matrix}
   A{{\sharp}_{t}}B & X  \\
   X & A{{\sharp}_{1-t}}B  \\
\end{matrix} \right]$ is PPT. From \cite[Lemma 3.4]{2}, there are two unitaries $U, V\in {{\mathcal M}_{2n}}$ partitioned into equally sized matrices,
	\[U=\left[ \begin{matrix}
   {{U}_{11}} & {{U}_{12}}  \\
   {{U}_{21}} & {{U}_{22}}  \\
\end{matrix} \right]\text{ and }V=\left[ \begin{matrix}
   {{V}_{11}} & {{V}_{12}}  \\
   {{V}_{21}} & {{V}_{22}}  \\
\end{matrix} \right]\]
such that
	$$\left[ \begin{matrix}
   A{{\sharp}_{t}}B & X  \\
   {{X}^{*}} & A{{\sharp}_{1-t}}B  \\
\end{matrix} \right]=U\left[ \begin{matrix}
   A{{\sharp}_{t}}B & O  \\
   O & O  \\
\end{matrix} \right]{{U}^{*}}+V\left[ \begin{matrix}
   O & O  \\
   O & A{{\sharp}_{1-t}}B  \\
\end{matrix} \right]{{V}^{*}}.$$
Hence,
	$$\left[ \begin{matrix}
   A{{\sharp}_{t}}B & X  \\
   {{X}^{*}} & A{{\sharp}_{1-t}}B  \\
\end{matrix} \right]=\widetilde{U}\left( A{{\sharp}_{t}}B \right){{\widetilde{U}}^{*}}+\widetilde{V}\left( A{{\sharp}_{1-t}}B \right){{\widetilde{V}}^{*}}$$
where
	\[\widetilde{U}=\left[ \begin{matrix}
   {{U}_{11}}  \\
   {{U}_{21}}  \\
\end{matrix} \right]\text{ and }\widetilde{V}=\left[ \begin{matrix}
   {{V}_{12}}  \\
   {{V}_{22}}  \\
\end{matrix} \right]\]
are isometries. This completes the proof.
\end{proof}

 Theorem \ref{12} implies the following remarkable result.
\begin{corollary}\label{13}
Let $A,B,X\in {{\mathcal M}_{n}}$  be such that $\left[ \begin{matrix}
   A & {{X}^{*}}  \\
   X & B  \\
\end{matrix} \right]$ is PPT. Then for any $0\le t\le 1$,
\[\left\| \left[ \begin{matrix}
   A{{\sharp}_{t}}B & X  \\
   {{X}^{*}} & A{{\sharp}_{1-t}}B  \\
\end{matrix} \right] \right\|\le \left\| A{{\sharp}_{t}}B \right\|+\left\| A{{\sharp}_{1-t}}B \right\|.\]
In particular
\[\left\| \left[ \begin{matrix}
   A\sharp B & X  \\
   {{X}^{*}} & A\sharp B  \\
\end{matrix} \right] \right\|\le 2\left\| A\sharp B \right\|.\]
\end{corollary}

\begin{remark}
Ando \cite[Theorem 3.3]{5} proved that if $\left[ \begin{matrix}
   A & {{X}^{*}}  \\
   X & B  \\
\end{matrix} \right]$ is PPT, then
\[\left\| X \right\|\le \left\| A\sharp B \right\|.\]
We know that \cite{8} if $\left[ \begin{matrix}
   A & {{X}^{*}}  \\
   X & B  \\
\end{matrix} \right]\ge O$, then
\[2\left\| X \right\|\le \left\| \left[ \begin{matrix}
   A & {{X}^{*}}  \\
   X & B  \\
\end{matrix} \right] \right\|.\]
Consequently, if $\left[ \begin{matrix}
   A & {{X}^{*}}  \\
   X & B  \\
\end{matrix} \right]$ is PPT, then 
\[\left\| X \right\|\le \frac{1}{2}\left\| \left[ \begin{matrix}
   A\sharp B & X  \\
   {{X}^{*}} & A\sharp B  \\
\end{matrix} \right] \right\|\le \left\| A\sharp B \right\|.\]
\end{remark}

\begin{remark}
It is well-known that
\[{{\left| \left\langle T{{\left| T \right|}^{\alpha +\beta -1}}x,y \right\rangle  \right|}^{2}}\le \left\langle {{\left| T \right|}^{2\alpha }}x,x \right\rangle \left\langle {{\left| {{T}^{*}} \right|}^{2\beta }}y,y \right\rangle ;\left( \alpha ,\beta \in \left[ 0,1 \right],\alpha +\beta \ge 1 \right)\]
for any $T$ \cite{14}. By \cite[Lemma 1]{12}, we infer that 
$$\left[ \begin{matrix}
   {{\left| T \right|}^{2\alpha }} & {{\left| T \right|}^{\alpha +\beta -1}}{{T}^{*}}  \\
   T{{\left| T \right|}^{\alpha +\beta -1}} & {{\left| {{T}^{*}} \right|}^{2\beta }}  \\
\end{matrix} \right]\ge O.$$
So, 
\[\begin{aligned}
   2\left\| T{{\left| T \right|}^{\alpha +\beta -1}} \right\|&\le \left\| \left[ \begin{matrix}
   {{\left| T \right|}^{2\alpha }} & {{\left| T \right|}^{\alpha +\beta -1}}{{T}^{*}}  \\
   T{{\left| T \right|}^{\alpha +\beta -1}} & {{\left| {{T}^{*}} \right|}^{2\beta }}  \\
\end{matrix} \right] \right\| \\ 
 & \le \left\| {{\left| T \right|}^{2\alpha }}+{{U}^{*}}{{\left| {{T}^{*}} \right|}^{2\beta }}U \right\| \\ 
 & =2\left\| {{\left| T \right|}^{2\alpha }}+{{\left| T \right|}^{2\beta }} \right\|  
\end{aligned}\]
i.e.,
\[\left\| T{{\left| T \right|}^{\alpha +\beta -1}} \right\|\le \frac{1}{2}\left\| \left[ \begin{matrix}
   {{\left| T \right|}^{2\alpha }} & {{\left| T \right|}^{\alpha +\beta -1}}{{T}^{*}}  \\
   T{{\left| T \right|}^{\alpha +\beta -1}} & {{\left| {{T}^{*}} \right|}^{2\beta }}  \\
\end{matrix} \right] \right\|\le \left\| {{\left| T \right|}^{2\alpha }}+{{\left| T \right|}^{2\beta }} \right\|.\]
In particular,
\[\left\| T \right\|=\frac{1}{2}\left\| \left[ \begin{matrix}
   \left| {{T}^{*}} \right| & T  \\
   {{T}^{*}} & \left| T \right|  \\
\end{matrix} \right] \right\|.\]
\end{remark}

\begin{theorem}
Let $A,B,X\in {{\mathcal M}_{n}}$ be such that $\left[ \begin{matrix}
   A & {{X}^{*}}  \\
   X & B  \\
\end{matrix} \right]\ge O$. Then
\[\left\| \left[ \begin{matrix}
   A & {{X}^{*}}  \\
   X & B  \\
\end{matrix} \right] \right\|\le \left\| A+B+\left| X \right|+\left| {{X}^{*}} \right| \right\|.\]
\end{theorem}
\begin{proof}
Indeed,
\[\begin{aligned}
   \left\| \left[ \begin{matrix}
   A & {{X}^{*}}  \\
   X & B  \\
\end{matrix} \right] \right\|&=\left\| \left[ \begin{matrix}
   A & O  \\
   O & B  \\
\end{matrix} \right]+\left[ \begin{matrix}
   O & {{X}^{*}}  \\
   X & O  \\
\end{matrix} \right] \right\| \\ 
 & \le \left\| \left[ \begin{matrix}
   A & O  \\
   O & B  \\
\end{matrix} \right]+\left[ \begin{matrix}
   \left| X \right| & O  \\
   O & \left| {{X}^{*}} \right|  \\
\end{matrix} \right] \right\| \\ 
 & =\left\| \left[ \begin{matrix}
   A+\left| X \right| & O  \\
   O & B+\left| {{X}^{*}} \right|  \\
\end{matrix} \right] \right\| \\ 
 & \le \left\| A+B+\left| X \right|+\left| {{X}^{*}} \right| \right\|  
\end{aligned}\]
where the first inequality follows from the fact that $\left[ \begin{matrix}
   O & {{X}^{*}}  \\
   X & O  \\
\end{matrix} \right]$ is Hermitian, and for any Hermitiam matrix $T$ , we have   $T\le \left| T \right|$. The second inequality is also obtained from \eqref{9}. This completes the proof.
\end{proof}

\section{Applications towards hyponormal and $(\alpha,\beta)$-normal matrices}
This section presents several results on hyponormal, semi-hyponormal, and $(\alpha,\beta)$-normal matrices. While some of these results can be considered as applications of the results of the previous section, other results are related but independent.

\begin{theorem}\label{1}
Let $T\in\mathcal{M}_n$. Then $\left[ \begin{matrix}
   \left| T \right| & {{T}^{*}}  \\
   T & \left| T \right|  \\
\end{matrix} \right]\ge O$ if and only if $T$ is semi-hyponormal.
\end{theorem}
\begin{proof}
It is easy to show that \cite{1} if $T$ is  semi-hyponormal, then 
	\[\left[ \begin{matrix}
   \left| T \right| & {{T}^{*}}  \\
   T & \left| T \right|  \\
\end{matrix} \right]\ge O.\]
We show that if $\left[ \begin{matrix}
   \left| T \right| & {{T}^{*}}  \\
   T & \left| T \right|  \\
\end{matrix} \right]\ge O$, then $T$ is semi-hyponormal. Indeed, if $T=U\left| T \right|$ is the polar decomposition of $T$, then
	$$\left[ \begin{matrix}
   \left| T \right| & {{T}^{*}}U  \\
   {{U}^{*}}T & {{U}^{*}}\left| T \right|U  \\
\end{matrix} \right]=\left[ \begin{matrix}
   I & O  \\
   O & {{U}^{*}}  \\
\end{matrix} \right]\left[ \begin{matrix}
   \left| T \right| & {{T}^{*}}  \\
   T & \left| T \right|  \\
\end{matrix} \right]\left[ \begin{matrix}
   I & O  \\
   O & U  \\
\end{matrix} \right]\ge O.$$
Since ${{T}^{*}}U={{U}^{*}}T=\left| T \right|$, we get
$$\left[ \begin{matrix}
   \left| T \right| & \left| T \right|  \\
   \left| T \right| & {{U}^{*}}\left| T \right|U  \\
\end{matrix} \right]\ge O.$$
From \eqref{4}, we get 
	$$\left| T \right|\le \left| T \right|\sharp{{U}^{*}}\left| T \right|U.$$
By the definition of geometric mean, we have
	$$\left| T \right|\le {{\left| T \right|}^{\frac{1}{2}}}{{\left( {{\left| T \right|}^{-\frac{1}{2}}}{{U}^{*}}\left| T \right|U{{\left| T \right|}^{-\frac{1}{2}}} \right)}^{\frac{1}{2}}}{{\left| T \right|}^{\frac{1}{2}}}.$$
Multiplying both sides by ${{\left| T \right|}^{-\frac{1}{2}}}$, we infer that
	$$I\le {{\left( {{\left| T \right|}^{-\frac{1}{2}}}{{U}^{*}}\left| T \right|U{{\left| T \right|}^{-\frac{1}{2}}} \right)}^{\frac{1}{2}}}.$$
This implies
$$\left| T \right|\le {{U}^{*}}\left| T \right|U.$$
Thus,
	$$\left| {{T}^{*}} \right|=U\left| T \right|{{U}^{*}}\le \left| T \right|,$$
which shows that $T$ is semi-hyponormal. This completes the proof.
\end{proof}

In \cite[(2.11)]{2}, it has been shown that if $A,B$ are normal,  then 
\begin{equation}
|A+B|\leq \frac{|A|+|B|+U^*(|A|+|B|)U}{2},
\end{equation}
where $U$ is the unitary matrix in the polar decomposition of $A+B.$ Since every normal matrix is necessarily semi-hyponormal, and because of \eqref{eq_amgm}, the following result significantly improves \cite[(2.11)]{2}.
\begin{corollary}
Let $A, B\in \mathcal{M}_n$ be semi-hyponormal and let $U$ be the unitary part in the polar decomposition $A+B=U\left| A+B \right|$. Then 
	\[\left| A+B \right|\le \left( \left| A \right|+\left| B \right| \right)\sharp\left({{U}^{*}}\left( \left| A \right|+\left| B \right| \right)U\right).\]
\end{corollary}
\begin{proof}
Using Theorem \ref{1}, we see that
\begin{equation}\label{3}
\left[ \begin{matrix}
   \left| A \right|+\left| B \right| & {{A}^{*}}+{{B}^{*}}  \\
   A+B & \left| A \right|+\left| B \right|  \\
\end{matrix} \right]\ge O.
\end{equation}
By \eqref{3} and Theorem \ref{5}, we get the desired result.
\end{proof}

\begin{corollary}\label{14}
Let $A,B\in \mathcal{M}_n$ be semi-hyponormal. Then
\[\left\| A+B \right\|\le \left\| \left| A \right|+\left| B \right| \right\|.\]
\end{corollary}

\begin{remark}
We highlight that Corollary \ref{14} is well-known for normal operators $A,B$ \cite[(1.42)]{4}.
\end{remark}

\begin{theorem}\label{25}
Let $T\in \mathcal{M}_n$. Then the following assertions are equivalent.
\begin{itemize}
\item[(i)] $T$ is $(\alpha,\beta)$-normal.
\medskip
\item[(ii)] $\left[ \begin{matrix}
   \frac{1}{\alpha }{{\left| {{T}^{*}} \right|}^{2}} & {{\left| T \right|}^{2}}  \\
   {{\left| T \right|}^{2}} & \frac{1}{\alpha }{{\left| T \right|}^{2}}  \\
\end{matrix} \right]\ge O$ and $\left[ \begin{matrix}
   \beta {{\left| T \right|}^{2}} & {{\left| {{T}^{*}} \right|}^{2}}  \\
   {{\left| {{T}^{*}} \right|}^{2}} & \beta {{\left| {{T}^{*}} \right|}^{2}}  \\
\end{matrix} \right]\ge O$.
\medskip
\item[(iii)] $\left[ \begin{matrix}
   \frac{1}{{{\alpha }^{2}}}{{\left| {{T}^{*}} \right|}^{2}} & {{\left| T \right|}^{2}}  \\
   {{\left| T \right|}^{2}} & \frac{1}{{{\alpha }^{2}}}{{\left| {{T}^{*}} \right|}^{2}}  \\
\end{matrix} \right]\ge O$ and $\left[ \begin{matrix}
   {{\beta }^{2}}{{\left| T \right|}^{2}} & {{\left| {{T}^{*}} \right|}^{2}}  \\
   {{\left| {{T}^{*}} \right|}^{2}} & {{\beta }^{2}}{{\left| T \right|}^{2}}  \\
\end{matrix} \right]\ge O$.

\end{itemize}
\end{theorem}
\begin{proof}
$\left( i \right)\Leftrightarrow \left( ii \right)$ Since $T$ is $(\alpha,\beta)$-normal, we have
\[\begin{aligned}
   &{{\left| T \right|}^{2}}\le \frac{1}{{{\alpha }^{2}}}{{\left| {{T}^{*}} \right|}^{2}}\\ 
 & \text{ }\Leftrightarrow {{\left| T \right|}^{2}}{{\left( \frac{1}{\alpha }{{\left| T \right|}^{2}} \right)}^{-1}}{{\left| T \right|}^{2}}\le \frac{1}{\alpha }{{\left| {{T}^{*}} \right|}^{2}} \\ 
 & \Leftrightarrow \left[ \begin{matrix}
   \frac{1}{\alpha }{{\left| {{T}^{*}} \right|}^{2}} & {{\left| T \right|}^{2}}  \\
   {{\left| T \right|}^{2}} & \frac{1}{\alpha }{{\left| T \right|}^{2}}  \\
\end{matrix} \right]\ge O  \quad \text{(by \cite[Theorem 1.3.3]{11})}.
\end{aligned}\]
Again, since $T$ is $(\alpha,\beta)$-normal, we have
\[\begin{aligned}
  & {{\left| {{T}^{*}} \right|}^{2}}\le {{\beta }^{2}}{{\left| T \right|}^{2}}\\ 
 & \text{ }\Leftrightarrow {{\left| {{T}^{*}} \right|}^{2}}{{\left( \beta {{\left| {{T}^{*}} \right|}^{2}} \right)}^{-1}}{{\left| {{T}^{*}} \right|}^{2}}\le \beta {{\left| T \right|}^{2}} \\ 
 & \Leftrightarrow \left[ \begin{matrix}
   \beta {{\left| T \right|}^{2}} & {{\left| {{T}^{*}} \right|}^{2}}  \\
   {{\left| {{T}^{*}} \right|}^{2}} & \beta {{\left| {{T}^{*}} \right|}^{2}}  \\
\end{matrix} \right]\ge O  \quad \text{(by \cite[Theorem 1.3.3]{11})}.
\end{aligned}\]
$\left( i \right)\Leftrightarrow \left( iii \right)$ See \cite[Theorem 2.2]{1}.
\end{proof}

We can establish the following theorem by employing the same arguments as in the proof of $\left( i \right)\Leftrightarrow \left( ii \right)$ in Theorem \ref{25}, but we present another proof.
\begin{theorem}\label{18}
Let $T\in \mathcal{M}_n$ be  $\left( \alpha ,\beta  \right)$-normal. Then 
\[\left[ \begin{matrix}
   \frac{1}{\sqrt{\alpha }}\left| {{T}^{*}} \right| & \left| T \right|  \\
   \left| T \right| & \frac{1}{\sqrt{\alpha }}\left| T \right|  \\
\end{matrix} \right]\ge O\text{ and }\left[ \begin{matrix}
   \sqrt{\beta }\left| T \right| & \left| {{T}^{*}} \right|  \\
   \left| {{T}^{*}} \right| & \sqrt{\beta }\left| {{T}^{*}} \right|  \\
\end{matrix} \right]\ge O.\]
\end{theorem}
\begin{proof}
According to the assumption,
	\[\left\langle \left| T \right|x,x \right\rangle \left\langle \left| {{T}^{*}} \right|y,y \right\rangle \le \frac{1}{\alpha }\left\langle \left| {{T}^{*}} \right|x,x \right\rangle \left\langle \left| {{T}^{*}} \right|y,y \right\rangle \]
and
	\[\left\langle \left| T \right|x,x \right\rangle \left\langle \left| {{T}^{*}} \right|y,y \right\rangle \le \beta \left\langle \left| T \right|x,x \right\rangle \left\langle \left| T \right|y,y \right\rangle \]
for any $x,y\in \mathcal H$. On the other hand, we know that (see, e.g., \cite[p. 216]{9})
	\[{{\left| \left\langle Tx,y \right\rangle  \right|}^{2}}\le \left\langle \left| T \right|x,x \right\rangle \left\langle \left| {{T}^{*}} \right|y,y \right\rangle .\]
Consequently,
	\[{{\left| \left\langle Tx,y \right\rangle  \right|}^{2}}\le \frac{1}{\alpha }\left\langle \left| {{T}^{*}} \right|x,x \right\rangle \left\langle \left| {{T}^{*}} \right|y,y \right\rangle \]
and
	\[{{\left| \left\langle Tx,y \right\rangle  \right|}^{2}}\le \beta \left\langle \left| T \right|x,x \right\rangle \left\langle \left| T \right|y,y \right\rangle .\]
The last two inequalities are equivalent to
\begin{equation}\label{11}
\left[ \begin{matrix}
   \frac{1}{\sqrt{\alpha }}\left| {{T}^{*}} \right| & {{T}^{*}}  \\
   T & \frac{1}{\sqrt{\alpha }}\left| {{T}^{*}} \right|  \\
\end{matrix} \right]\ge O\text{ and }\left[ \begin{matrix}
   \sqrt{\beta }\left| T \right| & {{T}^{*}}  \\
   T & \sqrt{\beta }\left| T \right|  \\
\end{matrix} \right]\ge O,
\end{equation}
thanks to \cite[Lemma 1]{12}.

Now assume that $T=U\left| T \right|$ is the polar decomposition of $T$. Then 
\begin{equation}\label{19}
\left[ \begin{matrix}
   I & O  \\
   O & {{U}^{*}}  \\
\end{matrix} \right]\left[ \begin{matrix}
   \frac{1}{\sqrt{\alpha }}\left| {{T}^{*}} \right| & {{T}^{*}}  \\
   T & \frac{1}{\sqrt{\alpha }}\left| {{T}^{*}} \right|  \\
\end{matrix} \right]\left[ \begin{matrix}
   I & O  \\
   O & U  \\
\end{matrix} \right]=\left[ \begin{matrix}
   \frac{1}{\sqrt{\alpha }}\left| {{T}^{*}} \right| & {{T}^{*}}U  \\
   {{U}^{*}}T & \frac{1}{\sqrt{\alpha }}{{U}^{*}}\left| {{T}^{*}} \right|U  \\
\end{matrix} \right]\ge O.
\end{equation}
On the other hand,
\[\left[ \begin{matrix}
   \sqrt{\beta }\left| T \right| & {{T}^{*}}  \\
   T & \sqrt{\beta }\left| T \right|  \\
\end{matrix} \right]\ge O\text{ }\Leftrightarrow \text{ }\left[ \begin{matrix}
   \sqrt{\beta }\left| T \right| & T  \\
   {{T}^{*}} & \sqrt{\beta }\left| T \right|  \\
\end{matrix} \right]\ge O.\]
So,
\begin{equation}\label{20}
\left[ \begin{matrix}
   I & O  \\
   O & U  \\
\end{matrix} \right]\left[ \begin{matrix}
   \sqrt{\beta }\left| T \right| & T  \\
   {{T}^{*}} & \sqrt{\beta }\left| T \right|  \\
\end{matrix} \right]\left[ \begin{matrix}
   I & O  \\
   O & {{U}^{*}}  \\
\end{matrix} \right]=\left[ \begin{matrix}
   \sqrt{\beta }\left| T \right| & T{{U}^{*}}  \\
   U{{T}^{*}} & \sqrt{\beta }U\left| T \right|{{U}^{*}}  \\
\end{matrix} \right]\ge O.
\end{equation}
One can easily check that ${{U}^{*}}T={{T}^{*}}U=\left| T \right|$. Meanwhile, $\left| {{T}^{*}} \right|=U\left| T \right|{{U}^{*}}$ (see \cite[p. 58]{9}), so ${{U}^{*}}\left| {{T}^{*}} \right|U=\left| T \right|$. Hence, by \eqref{19} and \eqref{20}, we obtain
\[\left[ \begin{matrix}
   \frac{1}{\sqrt{\alpha }}\left| {{T}^{*}} \right| & \left| T \right|  \\
   \left| T \right| & \frac{1}{\sqrt{\alpha }}\left| T \right|  \\
\end{matrix} \right]\ge O\text{ and }\left[ \begin{matrix}
   \sqrt{\beta }\left| T \right| & \left| {{T}^{*}} \right|  \\
   \left| {{T}^{*}} \right| & \sqrt{\beta }\left| {{T}^{*}} \right|  \\
\end{matrix} \right]\ge O,\]
as desired.
\end{proof}

The matrices in \eqref{11} are P.P.T. Therefore, by Corollary \ref{26}, we have the following eigenvalue inequalities.
\begin{corollary}\label{27}
Let $T\in \mathcal{M}_n$ be  $\left( \alpha ,\beta  \right)$-normal. Then 
	\[{{\lambda }_{j}}\left( 2\sqrt{\alpha }\left| T \right|-\left| {{T}^{*}} \right| \right)\le {{\lambda }_{j}}\left( \left| {{T}^{*}} \right| \right)\]
and
	\[{{\lambda }_{j}}\left( \frac{2}{\sqrt{\beta }}\left| {{T}^{*}} \right|-\left| T \right| \right)\le {{\lambda }_{j}}\left( \left| T \right| \right)\]
for $j=1,2,\ldots ,n$.
\end{corollary}

\begin{remark}
It is well-known that for any $T\in \mathcal{M}_n$,
\[\left\| \;\left| T \right|-\left| {{T}^{*}} \right| \;\right\|\le \left\| T \right\|.\]
From Corollary \ref{27}, we infer that if $T$ is  $\left( \alpha ,\beta  \right)$-normal, then
	\[\left\| \;2\sqrt{\alpha }\left| T \right|-\left| {{T}^{*}} \right| \;\right\|\le \left\| T \right\|,\]
and
	\[\left\| \;\frac{2}{\sqrt{\beta }}\left| {{T}^{*}} \right|-\left| T \right| \;\right\|\le \left\| T \right\|.\]
\end{remark}

The inequality \eqref{eq_amgm} is usually referred to as the operator arithmetic-geometric mean inequality. It is of great interest in the literature to find possible reverses for this inequality. Usually, such reverses are found under additional conditions, as seen in \cite{fms, gms}. In the following, we present a reverse of \eqref{eq_amgm} for $(\alpha,\beta)$-normal matrices.
\begin{proposition}
Let $T\in \mathcal{M}_n$ be  $\left( \alpha ,\beta  \right)$-normal. Then 
\[\frac{\left| T \right|+\left| {{T}^{*}} \right|}{2}\le \min \left\{ \frac{1}{\sqrt{\alpha }},\sqrt{\beta } \right\}\left( \left| T \right|\sharp\left| {{T}^{*}} \right| \right).\]
\end{proposition}
\begin{proof}
Theorem \ref{18} implies
\begin{equation}\label{21}
\left| T \right|\le \frac{1}{\sqrt{\alpha }}\left( \left| T \right|\sharp\left| {{T}^{*}} \right| \right)\text{ and }\left| {{T}^{*}} \right|\le \sqrt{\beta }\left( \left| T \right|\sharp\left| {{T}^{*}} \right| \right)
\end{equation}
due to \eqref{4}. Further, 
\begin{equation}\label{22}
\left| {{T}^{*}} \right|\le \frac{1}{\sqrt{\alpha }}\left( \left| T \right|\sharp\left| {{T}^{*}} \right| \right)\text{ and }\left| T \right|\le \sqrt{\beta }\left( \left| T \right|\sharp\left| {{T}^{*}} \right| \right)
\end{equation}
by utilizing the same approach as in the proof of inequality \eqref{15}.
Inequalities \eqref{21} and \eqref{22} say that
\[\left| T \right|\le \min \left\{ \frac{1}{\sqrt{\alpha }},\sqrt{\beta } \right\}\left( \left| T \right|\sharp \left| {{T}^{*}} \right| \right)\text{ and }\left| {{T}^{*}} \right|\le \min \left\{ \frac{1}{\sqrt{\alpha }},\sqrt{\beta } \right\}\left( \left| T \right|\sharp \left| {{T}^{*}} \right| \right).\] 
Adding the above two inequalities together implies the desired result.
\end{proof}

\begin{remark}
Inequalities \eqref{21} and \eqref{22} can be shown in another way. 
Since $f\left( t \right)=\sqrt{t}$ is operator monotone on $\left( 0,\infty  \right)$, and since $\alpha^2|T|^2\leq |T^*|^2,$ we infer that
\[\left| T \right|\le \frac{1}{\alpha }\left| {{T}^{*}} \right|.\]
This implies 
	\[\left| T \right|\le \frac{1}{\sqrt{\alpha }}\left( \left| T \right|\sharp\left| {{T}^{*}} \right| \right),\] where we have used the fact that if $A,B,C,D>O$ are such that $A\leq B$ and $C\leq D$, then $A\sharp C\leq B\sharp D.$

\end{remark}

For the following result, we remind the reader of positive linear maps. A linear map $\Phi:\mathcal{M})_n\to\mathcal{M}_n$ is said to be positive if $\Phi(A)\geq O$ whenever $A\geq O.$

\begin{theorem}
Let $T\in \mathcal{M}_n$ be  $\left( \alpha ,\beta  \right)$-normal  and let $\Phi $ be a positive linear map. If $\Phi \left( T \right)=U\left| \Phi \left( T \right) \right|$ is the polar decomposition of $\Phi \left( T \right)$, then 
\[\left| \Phi \left( T \right) \right|\le \frac{1}{\sqrt{\alpha }}\left( \Phi \left( \left| {{T}^{*}} \right| \right)\sharp{{U}^{*}}\Phi \left( \left| {{T}^{*}} \right| \right)U \right)\]	
and
\[\left| \Phi \left( T \right) \right|\le \sqrt{\beta }\left( \Phi \left( \left| T \right| \right)\sharp{{U}^{*}}\Phi \left( \left| T \right| \right)U \right).\]
\end{theorem}
\begin{proof}
First notice that every positive linear map is adjoint-preserving; i.e., ${{\Phi }^{*}}\left( T \right)=\Phi \left( {{T}^{*}} \right)$ for all $T$ \cite[Lemma 2.3.1]{4}.
It follows from \eqref{11} that
\[\left[ \begin{matrix}
   \frac{1}{\sqrt{\alpha }}\Phi \left( \left| {{T}^{*}} \right| \right) & {{\Phi }^{*}}\left( T \right)  \\
   \Phi \left( T \right) & \frac{1}{\sqrt{\alpha }}\Phi \left( \left| {{T}^{*}} \right| \right)  \\
\end{matrix} \right]\ge O\text{ and }\left[ \begin{matrix}
   \sqrt{\beta }\Phi \left( \left| T \right| \right) & {{\Phi }^{*}}\left( T \right)  \\
   \Phi \left( T \right) & \sqrt{\beta }\Phi \left( \left| T \right| \right)  \\
\end{matrix} \right]\ge O\]
thanks to \cite[Exercise 3.2.2]{4}. We get the desired result by mimicking the technique of the proof of Theorem \ref{5}.
\end{proof}

The following result presents an interesting reverse of the well-known inequality $\|T^2\|\leq \|T\|^2,$ for any $T$. We recall that a contraction $K$ satisfies $KK^*\leq I$, the identity. We also recall that the spectral radius $r(X)$ coincides with the operator norm $\|X\|$ when $X\geq O.$ 
\begin{theorem}\label{24}
Let $T\in \mathcal{M}_n$ be  $\left( \alpha ,\beta  \right)$-normal. Then 
\[{{\left\| T \right\|}^{2}}\le \frac{1}{\alpha }\left\| {{T}^{2}} \right\|\text{ and }{{\left\| T \right\|}^{2}}\le \beta \left\| {{T}^{2}} \right\|.\]
Indeed,
\[{{\left\| T \right\|}^{2}}\le \min \left\{ \frac{1}{\alpha },\beta  \right\}\left\| {{T}^{2}} \right\|.\]
\end{theorem}
\begin{proof}
Ando \cite{13} proved that $\left[ \begin{matrix}
   A & X  \\
   {{X}^{*}} & B  \\
\end{matrix} \right]\ge O$ if and only if there exists a contraction $K$ such that $X={{A}^{\frac{1}{2}}}K{{B}^{\frac{1}{2}}}$. It has been revealed in the proof of Theorem \ref{18} that $\left[ \begin{matrix}
   \frac{1}{\sqrt{\alpha }}\left| {{T}^{*}} \right| & \left| T \right|  \\
   \left| T \right| & \frac{1}{\sqrt{\alpha }}\left| T \right|  \\
\end{matrix} \right]\ge O$. Therefore, there exists a contraction $K$ such that $\left| T \right|=\frac{1}{\sqrt{\alpha }}{{\left| {{T}^{*}} \right|}^{\frac{1}{2}}}K{{\left| T \right|}^{\frac{1}{2}}}\ge O$.
So, we have
\[\begin{aligned}
   \left\| \;\left| T \right| \;\right\|&=\left\| T \right\| \\ 
 & =\frac{1}{\sqrt{\alpha }}\left\| {{\left| {{T}^{*}} \right|}^{\frac{1}{2}}}K{{\left| T \right|}^{\frac{1}{2}}} \right\| \\ 
 & =\frac{1}{\sqrt{\alpha }}r\left( {{\left| {{T}^{*}} \right|}^{\frac{1}{2}}}K{{\left| T \right|}^{\frac{1}{2}}} \right) \quad \text{(since $r\left( X \right)=\left\| X \right\|$ for positive $X$)}\\ 
 & =\frac{1}{\sqrt{\alpha }}r\left( K{{\left| T \right|}^{\frac{1}{2}}}{{\left| {{T}^{*}} \right|}^{\frac{1}{2}}} \right) \quad \text{(since $r\left( XY \right)=r\left( YX \right)$)}\\ 
 & \le \frac{1}{\sqrt{\alpha }}\left\| K{{\left| T \right|}^{\frac{1}{2}}}{{\left| {{T}^{*}} \right|}^{\frac{1}{2}}} \right\| \quad \text{(since $r\left( X \right)\le \left\| X \right\|$)}\\ 
 & \le \frac{1}{\sqrt{\alpha }}\left\| K \right\|\left\| {{\left| T \right|}^{\frac{1}{2}}}{{\left| {{T}^{*}} \right|}^{\frac{1}{2}}} \right\| \quad \text{(by the submultiplicative property of usual operator norm)}\\ 
 & \le \frac{1}{\sqrt{\alpha }}\left\| {{\left| T \right|}^{\frac{1}{2}}}{{\left| {{T}^{*}} \right|}^{\frac{1}{2}}} \right\| \quad \text{(since $K$ is contraction)}\\ 
 &\le \frac{1}{\sqrt{\alpha }}{{\left\| \left| T \right|\left| {{T}^{*}} \right| \right\|}^{\frac{1}{2}}}\quad \text{(by \cite[Theorem IX.2.1]{4})}\\
 & =\frac{1}{\sqrt{\alpha }}\left\| {{T}^{2}} \right\|^{\frac{1}{2}}.  
\end{aligned}\]
The second inequality comes from $\left[ \begin{matrix}
   \sqrt{\beta }\left| T \right| & \left| {{T}^{*}} \right|  \\
   \left| {{T}^{*}} \right| & \sqrt{\beta }\left| {{T}^{*}} \right|  \\
\end{matrix} \right]\ge O$
 and the same method as above. This completes the proof.
\end{proof}

\begin{remark}
We will give another method to prove Theorem \ref{24}. The operator inequality
\[\left| T \right|\le \frac{1}{\sqrt{\alpha }}\left( \left| T \right|\sharp\left| {{T}^{*}} \right| \right)\]
implies the following norm inequality
\[\left\| T \right\|\le \frac{1}{\sqrt{\alpha }}\left\| \left| T \right|\sharp\left| {{T}^{*}} \right| \right\|.\]
But, for any positive operators $A,B$, we know that
\[\left\| A\sharp B \right\|\le \left\| {{A}^{\frac{1}{2}}}{{B}^{\frac{1}{2}}} \right\|,\]
i.e.,
\[\left\| T \right\|\le \frac{1}{\sqrt{\alpha }}\left\| {{\left| T \right|}^{\frac{1}{2}}}{{\left| {{T}^{*}} \right|}^{\frac{1}{2}}} \right\|=\frac{1}{\sqrt{\alpha }}\left\| {{T}^{2}} \right\|^{\frac{1}{2}}.\]
\end{remark}

\subsection*{Declarations}
\begin{itemize}
\item {\bf{Availability of data and materials}}: Not applicable.
\item {\bf{Competing interests}}: The authors declare that they have no competing interests.
\item {\bf{Funding}}: Not applicable.
\item {\bf{Authors' contributions}}: Authors declare that they have contributed equally to this paper. All authors have read and approved this version.
\end{itemize}

\vskip 0.3 true cm

{\tiny (H. R. Moradi) Department of Mathematics, Mashhad Branch, Islamic Azad University, Mashhad, Iran
	
\textit{E-mail address:} hrmoradi@mshdiau.ac.ir}

\vskip 0.3 true cm 	

{\tiny (I. H. G\"um\"u\c s) Department of Mathematics, Faculty of Arts and Sciences, Ad\i yaman University, Ad\i yaman, Turkey}

{\tiny \textit{E-mail address:} igumus@adiyaman.edu.tr}

\vskip 0.3 true cm 	

{\tiny (M. Sababheh) Department of Basic Sciences, Princess Sumaya University for Technology, Amman, Jordan}
	
{\tiny\textit{E-mail address:} sababheh@psut.edu.jo}

\end{document}